\documentclass[reqno,11pt]{amsart}

\usepackage[a4paper,margin=24mm]{geometry}
\usepackage{dsfont}
\usepackage{mathtools}
\usepackage[foot]{amsaddr}
\usepackage[noadjust]{cite}
\usepackage[shortlabels]{enumitem}

\newcommand{\real}{\mathbb{R}}

\newcommand{\ind}{\mathds{1}}
\newcommand{\entropy}{\mathbf H}
\newcommand{\eqd}{\overset{d}{=}}

\newcommand{\abs}[1]{\left\lvert#1\right\rvert}

\newtheorem{theorem}{Theorem}[section]
\newtheorem{lemma}[theorem]{Lemma}

\theoremstyle{remark}

\newtheorem{remark}[theorem]{Remark}
\theoremstyle{definition}
\newtheorem{definition}[theorem]{Definition}

\numberwithin{equation}{section}

\allowdisplaybreaks

\begin{document}
\title[Entropies of CIR and Bessel processes]{Entropies of Cox--Ingersoll--Ross and Bessel processes as~functions of time and of related parameters}
\author[Ivan Kucha, Yuliya Mishura, and Kostiantyn Ralchenko]{Ivan Kucha$^1$, Yuliya Mishura$^{1}$, and Kostiantyn Ralchenko$^{1,2}$}
\address{$^1$Department of Probability, Statistics and Actuarial Mathematics, Taras Shevchenko National University of Kyiv,
Volodymyrska 64, 01601 Kyiv, Ukraine}
\address{$^2$School of Technology and Innovations, University of Vaasa,
P.O. Box 700, Vaasa, FIN-65101, Finland}
\thanks{Y.M. is supported by The Swedish Foundation for Strategic Research, grant UKR24-0004, and by the Japan Science and Technology Agency CREST, project reference number JPMJCR2115.}
\thanks{K.R. is supported by the Research Council of Finland, decision
number 367468.}
\email{ivan.kucha@knu.ua, yuliyamishura@knu.ua, kostiantynralchenko@knu.ua}
\begin{abstract}
We investigate the long-time asymptotic behavior of various entropy measures associated with the Cox--Ingersoll--Ross (CIR) and squared Bessel processes. As the one-dimensional distributions of both processes follow noncentral chi-squared laws, we first derive sufficient conditions for the existence of these entropy measures for a noncentral chi-squared random variable. We then analyze their limiting behavior as the noncentrality parameter approaches zero and apply these results to the Cox--Ingersoll--Ross and squared Bessel processes.
We prove that, as time tends to infinity, the entropies of the CIR process converge to those of its stationary distribution, while for the squared Bessel process, the Shannon, Rényi, and generalized Rényi entropies diverge, however, the Tsallis and Sharma–Mittal entropies may diverge or remain finite depending on the entropy parameters. Finally, we demonstrate that, as the CIR process converges to the squared Bessel process, the corresponding entropies also converge.
\end{abstract}

\date{\today}

\maketitle

\section{Introduction}

The entropy of a probability distribution, introduced by Claude Shannon, has very peculiar properties. For example, the entropy of a continuous probability distribution is not, generally speaking, the limit of the entropies of discrete distributions converging to it. 
It is all the more interesting to study, so to speak, the ``positive'' properties of Shannon entropy and other entropies introduced later, in particular, Rényi, generalized Rényi, Tsallis, and Sharma--Mittal entropies. 
In particular, it is interesting to study the properties of these entropies with respect to the parameters of the underlying distributions. Some attempts in this direction are presented in papers \cite{bodnarchuk2024properties, BraimanEtAl_NAMC2024, MalyarenkoEtAl_Axioms2023, MishuraEtAl_VMSTA2024, nadarajah2003renyi, nishiyama2020relations}. 

It is even more interesting to study the properties of the entropies of stochastic processes, as is done, e.g., in \cite{gomez2008entropy,Cattiaux23,follmer85,Cincotta23,orlando2023expecting, degregorio2009renyi,Zhang2025} (we do not pretend to be an exhaustive list, as there are many works  on this topic). In this case, of course, the question arises of how to most adequately represent the entropy of a stochastic process. A natural answer is to study the entropy of finite-dimensional distributions. 

For Gaussian vectors and processes, Shannon's entropy is discussed in detail in the book~\cite{stratonovich}. But even in this relatively simple case, Shannon's entropy, which is expressed through the logarithm of the determinant of the covariance matrix, has rather complex and difficult-to-verify analytical properties, which is clearly seen in the example of such a widespread process as fractional Brownian motion (\cite{MMRS2023}). 
Therefore, when considering more complicated fractional Gaussian processes, in the paper \cite{MalyarenkoEtAl_Axioms2023} we limited ourselves to the entropies of one-dimensional distributions. The situation is even more complicated with the entropies of stochastic processes that have non-Gaussian distributions. 

However, even in these cases, it is possible to investigate the behavior of entropies both as functions of time and as functions of the process's parameters. The present paper is devoted to the study of such problems.
More precisely, we consider two closely related stochastic processes: the Cox--Ingersoll--Ross (CIR) process and the squared Bessel process. Both are strictly positive solutions of stochastic differential equations, making them well-suited for modeling real-world phenomena in physics, biology, and economics. In financial mathematics, the CIR process is widely used for modeling interest rate dynamics and in bond pricing frameworks; see, for example, \cite{Brigo,DiFrancesco22,Orlando19}. Notably, the squared Bessel process can, in a certain sense, be viewed as a limiting case of the CIR process under a phase transition \cite{MRK25}.

Our main interest lies in the asymptotic behavior of entropy measures for these processes as time $t\to\infty$. We show that, for the CIR process, all entropies converge to those of the corresponding stationary distribution, which is a gamma distribution. These limiting values are finite and can be computed explicitly. In contrast, for the squared Bessel process, the Shannon, Rényi, and generalized Rényi entropies diverge, whereas the Tsallis and Sharma–Mittal entropies may either diverge or converge to a finite limit, depending on the values of the entropy parameters.

It is known that the squared Bessel process arises as a limiting case of the Cox--Ingersoll--Ross process when the mean-reversion parameter $b$ tends to zero. Recent results in \cite[Section~IV]{MRK25} study this convergence for every fixed $t > 0$ in the $L^1$- and $L^2$-norms, and establish the corresponding rates of convergence. In the present paper, we extend the analysis by investigating the convergence of the corresponding entropy measures. Specifically, we show that as $b \to 0$, the entropies of the CIR process converge to those of the squared Bessel process.

Since the one-dimensional distributions of both processes follow scaled noncentral chi-squared laws, we first derive sufficient conditions for the existence of entropy measures for a noncentral chi-squared random variable. We also analyze their limiting behavior as the noncentrality parameter tends to zero.

The paper is organized as follows.
Section \ref{sec:preliminaries} provides a review of the central and noncentral chi-squared distributions and the relationships between them. This section also introduces six entropies (Shannon, R\'enyi, two generalized R\'enyi, Tsallis, and Sharma--Mittal) and examines the effect of scaling on their values.
In Section~\ref{sec:conv-shi-sq}, we establish the existence and convergence of the entropies of the noncentral chi-squared distribution as the noncentrality parameter approaches zero.
These results are then applied in Section \ref{sec:cir-bessel} to the entropies of the CIR and squared Bessel processes; this section presents our main findings.
Appendix contains auxiliary results concerning the modified Bessel function of the first kind, which appears in the probability density function of the noncentral chi-squared distribution, as well as properties of the density itself. In particular, we derive several useful inequalities and convergence results for integrals that arise in the definitions of the entropies.

\section{Preliminaries} \label{sec:preliminaries}

We begin by fixing the notation and recalling the distributions and entropies under study. 

\subsection{Central and noncentral chi-squared distributions}

\begin{definition}[Central chi-squared distribution]
Let $k>0$. The (central) \emph{chi-squared distribution} with $k$ degrees of freedom is an absolutely continuous probability distribution whose probability density function (PDF) is given by
\begin{equation}\label{pdf-central}
    f_{X_k}(x) = \frac{1}{2^{k/2} \Gamma(k/2)} \, x^{k/2 - 1} e^{-x/2} \ind_{x > 0},
\end{equation}
where $\Gamma(x)$ denotes the gamma function.
\end{definition}
This distribution arises as the distribution of the sum of squares of $k$ independent standard normal random variables when $k$ is a positive integer, and more generally as a special case of the gamma distribution.
The following definition generalizes the central chi-squared law to account for a nonzero mean in the underlying normal random variables. 

\begin{definition}[Noncentral chi-squared distribution] \label{def:noncentralchisquared}
Let $k,\lambda>0$. The \emph{noncentral chi-squared distribution} with $k$ degrees of freedom and noncentrality parameter $\lambda$ is an absolutely continuous probability distribution whose PDF is given by
\begin{equation}\label{pdf-noncentral}
    f_{X_{k, \lambda}}(x) = \frac{1}{2} e^{-(x + \lambda)/2} \left(\frac{x}{\lambda}\right)^{k/4 - 1/2} I_{k/2 - 1}\bigl(\sqrt{\lambda x}\,\bigr)\ind_{x > 0},
\end{equation}
where $I_{\nu}(x)$ denotes the modified Bessel function of the first kind of order $\nu$, see Appendix~\ref{app:inequality}.
\end{definition}

For a comprehensive treatment of central and noncentral chi-squared distributions, we refer the reader to \cite[Chapter 18]{JKB-v1} and \cite[Chapter 29]{JKB-v2}, respectively. 

The PDF of the noncentral chi-squared distribution can be expressed as a Poisson mixture of central chi-squared PDFs \cite[formula (29.4)]{JKB-v2}:
\begin{equation}\label{mixure}
f_{X_{k,\lambda}}(x) = e^{-\lambda/2} \sum_{r=0}^\infty \frac{1}{r!} \left( \frac{\lambda}{2} \right)^r f_{X_{k + 2r}}(x),
\quad x > 0.
\end{equation}

It is known \cite[Chapter 29, Section 10]{JKB-v2} that when $\lambda = 0$, the noncentral chi-squared distribution reduces to the central chi-squared distribution. Specifically, for any $x > 0$ and $k > 0$,
\begin{equation} \label{chi_squared_convergence}
f_{X_{k,\lambda}}(x) \to f_{X_k}(x)
\quad \text{as } \lambda \to 0.
\end{equation}
This convergence follows directly from the mixture representation \eqref{mixure} by letting \mbox{$\lambda \to 0$}.
Alternatively, \eqref{chi_squared_convergence} can be deduced directly from the formulas for the central and noncentral PDFs by applying the asymptotic expansion
\[
I_{\nu}(x) = \frac{1}{\Gamma(\nu + 1)} \left(\frac{x}{2}\right)^\nu + o(x^\nu), \quad x \to 0,
\]
which follows from the series representation \eqref{bessel-func} of the modified Bessel function $I_\nu(x)$.

\subsection{Entropy functionals}

\begin{definition}
\label{def:entropies}
Let $X$ be a random variable and $f_X(x)$, $x \in \real$, be its PDF.
    \begin{enumerate}
    \item The \emph{Shannon entropy} is given by
    \[
    \entropy_S (X) = - \int_{\mathbb{R}} f_X(x) \log f_X(x) \, dx.
    \]
    
    \item The \emph{R\'enyi entropy} with parameter $\alpha > 0$, $\alpha \ne 1$, is given by
    \[
    \entropy_R(X;\alpha) = \frac{1}{1-\alpha} \log \int_{\mathbb{R}} f_X^\alpha(x) \, dx.
    \]

    \item The \emph{generalized R\'enyi entropy} in the case $\alpha \neq \beta$, $\alpha, \beta > 0$ is given by
    \[
    \entropy_{GR}(X;\alpha, \beta) = \frac{1}{\beta - \alpha} \log \frac{\int_{\mathbb{R}} f_X^\alpha(x) \, dx}{\int_{\mathbb{R}} f_X^\beta(x) \, dx}.
    \]
    The \emph{generalized R\'enyi entropy} in the case $\alpha = \beta > 0$ is given by
    \[
    \entropy_{GR}(X;\alpha) = - \frac{\int_{\mathbb{R}} f_X^\alpha(x) \log f_X(x) \, dx}{\int_{\mathbb{R}} f_X^\alpha(x) \, dx}.
    \]

    \item The \emph{Tsallis entropy} with index $\alpha > 0$, $\alpha \neq 1$ is given by
    \[
    \entropy_T(X;\alpha) = \frac{1}{1-\alpha} \left( \int_{\mathbb{R}} f_X^\alpha(x) \, dx - 1 \right).
    \]

    \item The \emph{Sharma--Mittal entropy} with positive indices $\alpha \neq 1$ and $\beta \neq 1$ is defined as
    \[
    \entropy_{SM}(X;\alpha, \beta) = \frac{1}{1-\beta} \left( \left( \int_{\mathbb{R}} f_X^\alpha(x) \, dx \right)^{\frac{1-\beta}{1-\alpha}} - 1 \right).
    \]
\end{enumerate}
\end{definition}

\begin{remark}
The formulas for the entropies of the central chi-squared distribution are derived in \cite[Proposition 3.6]{bodnarchuk2024properties}.
\end{remark}

In what follows, we consider the entropies of one-dimensional distributions arising from the CIR and Bessel processes. Each of these distributions can be represented as the law of a noncentral chi-squared random variable multiplied by a deterministic scaling factor. Consequently, it is necessary to relate the entropies of a scaled random variable $C\cdot X$ for some $ C>0$ to those of the original variable $X$. The following straightforward lemma provides the corresponding transformation rules.

\begin{lemma}[Scaling of entropy functionals]
\label{l:entropy-scaling}
Let $X$ be a random variable with probability density function $f_X$, and let $C > 0$ be a constant. Then the entropies of $C\cdot X$ are expressed via those of $X$ as follows:
\begin{align*}
\entropy_S(C\cdot X) &= \entropy_S(X) + \log C, \\
\entropy_R(C\cdot X;\alpha) &= \entropy_R(X;\alpha) + \log C
&& (\alpha > 0,\, \alpha \ne 1), \\
\entropy_{GR}(C\cdot X;\alpha,\beta) &= \entropy_{GR}(X;\alpha,\beta) + \log C
&& (\alpha, \beta > 0,\, \alpha \ne \beta),\\
\entropy_{GR}(C\cdot X;\alpha) &= \entropy_{GR}(X;\alpha) + \log C
&& (\alpha > 0),\\
\entropy_T(C\cdot X;\alpha) &= C^{1 - \alpha} \entropy_T(X;\alpha) + \dfrac{C^{1 - \alpha} - 1}{1 - \alpha}
&& (\alpha > 0,\, \alpha \ne 1),\\
\entropy_{SM}(C\cdot X;\alpha,\beta) &= C^{1 - \beta} \entropy_{SM}(X;\alpha,\beta) + \dfrac{C^{1 - \beta} - 1}{1 - \beta}
&& (\alpha, \beta > 0,\, \alpha \ne 1,\, \beta \ne 1).
\end{align*}
\end{lemma}

\begin{proof}
Let $f_X$ denote the density of $X$. Then the density of $Y =C\cdot X$ is given by
\[
f_Y(y) = \frac{1}{C} f_X\left(\frac{y}{C}\right).
\]
Substitute this expression into each entropy definition and change variables via $y = Cx$, $dy = C\,dx$. In the case of Shannon, Rényi, and both generalized Rényi entropies, the logarithmic term splits as $\log f_Y(y) = \log f_X(x) - \log C$, yielding an additive $\log C$ term. In the Tsallis and Sharma--Mittal entropies, the prefactor $C^{1 - \alpha}$ or $C^{1 - \beta}$ appears due to the power of the density and scaling of the integral, which results in the corresponding multiplicative and correction terms.
\end{proof}

\section{Convergence of the entropies of the noncentral chi-squared distribution as~noncentrality parameter tends to zero}
\label{sec:conv-shi-sq}

In this section, we study the limiting behavior of various entropies associated with the noncentral chi-squared distribution as the noncentrality parameter tends to zero. We provide sufficient conditions for the existence and convergence of these entropy measures.

\begin{theorem} \label{thm:EntropyConvergence}
Let $X_{k, \lambda}$ denote a noncentral chi-squared random variable with $k > 1$ degrees of freedom and noncentrality parameter $\lambda > 0$, and let $X_k$ denote the corresponding central chi-squared distribution with $k > 1$ degrees of freedom. Then the following statements hold:
\begin{enumerate}
\item
The Shannon entropy of $X_{k, \lambda}$ exists for all $\lambda > 0$ and $k > 1$, and converges to the Shannon entropy of $X_k$ as $\lambda \to 0$:
\[
\lim_{\lambda \to 0} \entropy_S(X_{k,\lambda}) = \entropy_S(X_k).
\]

\item
The Rényi entropy of $X_{k,\lambda}$ with parameter $\alpha > 0$, $\alpha \ne 1$, exists if and only if 
$k > 2 - \frac{2}{\alpha}$. For such values of $k$ and $\alpha$, the following limit holds:
\[
\lim_{\lambda \to 0} \entropy_R(X_{k,\lambda};\alpha) = \entropy_R(X_k;\alpha).
\]

\item
The generalized Rényi entropy of $X_{k, \lambda}$ with parameters $\alpha \ne \beta$ and $\alpha, \beta > 0$ exists if and only if
$k > 2 - \frac{2}{\alpha \vee \beta}$.
For all such values of $k$, $\alpha$, and $\beta$, the limit
\[
\lim_{\lambda \to 0} \entropy_{GR}(X_{k,\lambda};\alpha,\beta) = \entropy_{GR}(X_k;\alpha,\beta)
\]
holds.

If $\alpha = \beta > 0$, then the generalized Rényi entropy of $X_{k,\lambda}$ exists for all $k > 2 - \frac{2}{\alpha}$. In this case,
\[
\lim_{\lambda \to 0} \entropy_{GR}(X_{k,\lambda};\alpha) = \entropy_{GR}(X_k;\alpha).
\]

\item 
The Tsallis entropy of $X_{k,\lambda}$ with index $\alpha > 0$, $\alpha \ne 1$, exists if and only if $k > 2 - \frac{2}{\alpha}$. For such $k$ and $\alpha$, the following limit holds:
\[
\lim_{\lambda \to 0} \entropy_T(X_{k,\lambda};\alpha) = \entropy_T(X_k;\alpha).
\]

\item 
The Sharma--Mittal entropy of $X_{k, \lambda}$ with indices $\alpha \ne 1$, $\beta \ne 1$, and $\alpha, \beta > 0$, exists if and only if $k > 2 - \frac{2}{\alpha}$. For such $k$ and $\alpha$,
\[
\lim_{\lambda \to 0} \entropy_{SM}(X_{k,\lambda}; \alpha, \beta) = \entropy_{SM}(X_k;\alpha, \beta).
\]
\end{enumerate}
\end{theorem}

\begin{proof}
Let $f_{X_k}$ and $f_{X_{k,\lambda}}$ denote the PDFs of the central and noncentral chi-squared distributions, as defined in \eqref{pdf-central} and \eqref{pdf-noncentral}, respectively.

(1) The result for the Shannon entropy follows directly from the definition of $\entropy_S(X_{k,\lambda})$ and Lemma~\ref{l:integrals}~$(ii)$. 

(2)--(6) Similarly, by Definition~\ref{def:entropies}, all other entropy expressions depend continuously on the following integrals:
\[
\int_{\mathbb{R}} f_{X_{k,\lambda}}^\alpha(x) \, dx, \quad
\int_{\mathbb{R}} f_{X_{k,\lambda}}^\beta(x) \, dx, \quad
\int_{\mathbb{R}} f_{X_{k,\lambda}}^\alpha(x) \log f_{X_{k,\lambda}}(x) \, dx.
\]
Under the stated necessary and sufficient conditions, these integrals exist and converge to the corresponding integrals involving $f_{X_k}$ as $\lambda \to 0$, by Lemma~\ref{l:integrals}, whence the proof follows.
\end{proof}

\begin{remark}
The condition $k > 1$ in Theorem~\ref{thm:EntropyConvergence} is necessary for the application of Lemma~\ref{l:integrals}.  
Moreover, since $k > 1$, the additional condition $k > 2 - \frac{2}{\alpha}$ is only restrictive when $\alpha > 2$.
\end{remark}

\begin{remark}\label{rem:conv-chi-sq}
Let $X_{k,\lambda}$ and $X_{k,\lambda_0}$ denote noncentral chi-squared random variables with $k$ degrees of freedom and noncentrality parameters $\lambda$ and $\lambda_0$, respectively. Analogously to Theorem~\ref{thm:EntropyConvergence}, one can show that as $\lambda \to \lambda_0$, the entropies
\[
\entropy_S(X_{k,\lambda}),\;
\entropy_R(X_{k,\lambda}; \alpha),\;
\entropy_{GR}(X_{k,\lambda}; \alpha, \beta),\;
\entropy_{GR}(X_{k,\lambda}; \alpha),\;
\entropy_T(X_{k,\lambda}; \alpha),\;
\entropy_{SM}(X_{k,\lambda}; \alpha, \beta)
\]
converge to
\[
\entropy_S(X_{k,\lambda_0}),\;
\entropy_R(X_{k,\lambda_0}; \alpha),\;
\entropy_{GR}(X_{k,\lambda_0}; \alpha, \beta),\;
\entropy_{GR}(X_{k,\lambda_0}; \alpha),\;
\entropy_T(X_{k,\lambda_0}; \alpha),\;
\entropy_{SM}(X_{k,\lambda_0}; \alpha, \beta),
\]
respectively, under the same conditions on $k$, $\alpha$, and $\beta$ as specified in Theorem~\ref{thm:EntropyConvergence}.
\end{remark}


\section{Asymptotic behavior of entropies of Cox--Ingersoll--Ross and squared Bessel processes}
\label{sec:cir-bessel}

In this section, we apply the results of Section~\ref{sec:conv-shi-sq} to study the asymptotic behavior of the entropies associated with the one-dimensional distributions of two related stochastic processes: the Cox--Ingersoll--Ross (CIR) process and the squared Bessel process. An overview of these processes, along with their basic properties and relevant references, is provided in the recent paper \cite{MRK25}. For further details, we refer to the monograph \cite{RevuzYor}, the classical papers \cite{GJY03,PitmanYor}, and more recent contributions \cite{MiYu2, MIYu1}; see also \cite{BAK2013} for a comparative analysis of the two processes.  
We examine the asymptotic behavior of each process as $t \to \infty$, and also establish the convergence of the CIR process to the squared Bessel process in the limiting regime where the mean-reversion parameter $b \downarrow 0$.

\subsection{Asymptotic behavior of the entropies of the CIR process as $t \to \infty$}
\label{sec:cir}

The \emph{Cox--Ingersoll--Ross} (CIR) model, introduced in \cite{cox1985theory}, describes the risk-neutral evolution of the instantaneous risk-free interest rate via a mean-reverting square-root diffusion process. The square-root volatility term ensures the positivity of interest rates (under the Feller condition), while still allowing for closed-form expressions for bond prices and various interest-rate derivatives, making the CIR model a cornerstone of one-factor term structure modeling.

The dynamics of the CIR process are governed by the stochastic differential equation
\begin{equation}
  dr_t = \left(a - b r_t\right)\,dt + \sigma\sqrt{r_t}\,dW_t, \quad t \geq 0,
  \label{CIRSDE}
\end{equation}
where $r_0 > 0$, $a > 0$, $b > 0$, $\sigma > 0$, and $W = \{W_t, t \geq 0\}$ is a standard Wiener process. We assume that the Feller condition is satisfied, which ensures that the CIR process a.s.\ remains strictly positive; this condition is given by
\begin{equation}\label{feller}
2a \ge \sigma^2.
\end{equation}

The solution to the CIR equation has the following probability density function:
\begin{equation}\label{pdf-cir}
    p_t(x) = \frac{1}{2c(t)}\left(\frac{x}{r_0e^{-bt}} \right)^{\nu /2}\exp \left\{-\frac{x+r_0e^{-bt}}{2c(t)}\right\} I_\nu \left( \frac{e^{-bt/2}\sqrt{xr_0}}{c(t)} \right) \ind_{x>0},
\end{equation}
where 
\begin{equation}\label{cir-c-nu}
    c(t) = \frac{\sigma^2}{4b}\left( 1-e^{-bt} \right), \quad \nu = \frac{2a}{\sigma^2}-1.
\end{equation}

As $t \to \infty$, the density $p_t(x)$ converges to the Gamma distribution density:
\begin{equation}\label{cir-conv-to-gamma}
    p_t(x) \to \frac{(2b/\sigma^2)^{2a/\sigma^2}}{\Gamma(2a/\sigma^2)}x^{2a/\sigma^2-1}e^{-2bx/\sigma^2} \ind_{x>0} \eqqcolon p_\infty (x),
\end{equation}
which corresponds to the Gamma distribution with shape parameter $\frac{2a}{\sigma^2}$ and scale parameter $\frac{\sigma^2}{2b}$.

\begin{theorem}[Asymptotic entropy behavior for the CIR model]
\label{thm:CIR_entropy_convergence}
Let $r = \{r_t, t \ge 0\}$ be the CIR process \eqref{CIRSDE} with parameters $a > 0$, $b > 0$, $\sigma > 0$ that satisfy the condition \eqref{feller}, and let $r_\infty$ be a gamma-distributed random variable with PDF $p_\infty(x)$ defined in \eqref{cir-conv-to-gamma}. Then the following limits hold as $t \to \infty$:
\begin{align*}
  \entropy_S(r_t)
  &\to\entropy_S(r_\infty)
  = -\log \frac{2b}{\sigma^2} + \log \Gamma\left(\frac{2a}{\sigma^2}\right) + \frac{2a}{\sigma^2} - \psi\left(\frac{2a}{\sigma^2}\right) \left(\frac{2a}{\sigma^2}-1\right), 
  \\
  \entropy_R(r_t; \alpha)
  &\to\entropy_R(r_\infty; \alpha)
  = -\log \frac{2b}{\sigma^2} -\frac{1}{1-\alpha}
        \log\frac{\Gamma^\alpha\left(\frac{2a}{\sigma^2}\right)}{\Gamma\left(\alpha\left(\frac{2a}{\sigma^2}-1\right)+1\right)}
        -\frac{1-\alpha+\frac{2a\alpha}{\sigma^2}}{1-\alpha}\log\alpha,
  \\
  \entropy_{GR}(r_t; \alpha, \beta )
  &\to\entropy_{GR}(r_\infty; \alpha, \beta )\\
  &=-\log\frac{2b}{\sigma^2} + \frac{1}{\beta-\alpha}\log\left(
\frac{\beta^{\beta\left(\frac{2a}{\sigma^2}-1\right)+1}}{\alpha^{\alpha\left(\frac{2a}{\sigma^2}-1\right)+1}}
\frac{\Gamma\left(\alpha\left(\frac{2a}{\sigma^2}-1\right)+1\right)}{\Gamma\left(\beta\left(\frac{2a}{\sigma^2}-1\right)+1\right)}\right)
+\log\Gamma\left(\frac{2a}{\sigma^2}\right),
  \\
  \entropy_{GR}(r_t; \alpha)
  &\to\entropy_{GR}(r_\infty; \alpha)
  =-\log\frac{2b}{\sigma^2} +\log\Gamma\left(\frac{2a}{\sigma^2}\right)+\left(\frac{2a}{\sigma^2}-1\right)\log\alpha
  \\*
  &\qquad- \left(\frac{2a}{\sigma^2}-1\right)\psi\left(\alpha\left(\frac{2a}{\sigma^2}-1\right)+1\right)+\frac{2a}{\sigma^2}-1+\frac{1}{\alpha},
  \\
  \entropy_T(r_t; \alpha)
  &\to\entropy_T(r_\infty; \alpha)
  =\frac{1}{1-\alpha}\left(\left(\frac{2b}{\sigma^2}\right)^{\alpha-1}
    \frac{\alpha^{\alpha\left(1-\frac{2a}{\sigma^2}\right)-1}\Gamma\left(\alpha\left(\frac{2a}{\sigma^2} - 1\right)+1\right)}{\Gamma^\alpha\left(\frac{2a}{\sigma^2}\right)}-1\right),
  \\
  \entropy_{SM}(r_t; \alpha, \beta )
  &\to\entropy_{SM}(r_\infty; \alpha, \beta)
  \\
  &= \frac{1}{1-\beta}\left(
\frac{\left(\frac{2b}{\sigma^2}\right)^{\beta-1} \alpha^{\frac{\left(\alpha\left(1-\frac{2a}{\sigma^2}\right)-1\right)(1-\beta)}{1-\alpha}}}{\Gamma^{\frac{\alpha(1-\beta)}{1-\alpha}}\left(\frac{2a}{\sigma^2}\right)}
\Gamma^{\frac{1-\beta}{1-\alpha}}\left(\alpha\left(\frac{2a}{\sigma^2}-1\right)+1\right)
-1\right),
\end{align*}
where $\psi(\nu) = \frac{d}{d\nu} \log \Gamma(\nu)$ denotes the digamma function.
\end{theorem}

\begin{proof}
Comparing the expressions \eqref{pdf-noncentral} and \eqref{pdf-cir}, we observe that
\[
p_t(x) = \frac{1}{c(t)} f_{X_{2\nu+2, \lambda(t)}} \left(\frac{x}{c(t)}\right),
\]
where $f_{X_{2\nu+2, \lambda(t)}}$ is the PDF of the noncentral chi-squared distribution with $k = 2\nu+2$ degrees of freedom and noncentrality parameter
\[
\lambda(t) = \frac{r_0e^{-bt}}{c(t)},
\]
with $\nu$ and $c(t)$ defined in \eqref{cir-c-nu}.

Thus, the distribution of $r_t$ coincides with that of a noncentral chi-squared random variable scaled by $c(t)$:
\[
r_t \eqd c(t) X_{4a/\sigma^2, \lambda(t)}.
\]

As $t \to \infty$, we have
\[
c(t) \to \frac{\sigma^2}{4b}
\quad \text{and} \quad
\lambda(t) \to 0.
\]

Consider first the Shannon entropy. By Theorem~\ref{thm:EntropyConvergence} and Lemma~\ref{l:entropy-scaling}, we have
\begin{align*}
\entropy_S(r_t)
&= \entropy_S \left(c(t) X_{4a/\sigma^2, \lambda(t)}\right)
= \entropy_S \left(X_{4a/\sigma^2, \lambda(t)}\right)
+ \log c(t)
\\
&\to \entropy_S \left(X_{4a/\sigma^2}\right)
+ \log \frac{\sigma^2}{4b}
= \entropy_S \left(\frac{\sigma^2}{4b} X_{4a/\sigma^2}\right),
\quad \text{as } t \to \infty.
\end{align*}

It remains to observe that
\[
\frac{\sigma^2}{4b} X_{4a/\sigma^2} \eqd r_\infty,
\]
as can be verified by comparing their PDFs:
\[
\frac{4b}{\sigma^2} f_{X_{4a/\sigma^2}}\left(\frac{4b}{\sigma^2} x\right) = p_\infty(x).
\]
Hence, $\entropy_S(r_t) \to \entropy_S(r_\infty)$ as $t\to\infty$.

The convergence of the other entropy functionals of $r_t$ to those of $r_\infty$ follows by applying the same scaling argument and the convergence result from Theorem~\ref{thm:EntropyConvergence}.
Observe that, under the Feller condition \eqref{feller}, the number of degrees of freedom satisfies $k = \frac{4a}{\sigma^2} \ge 2$. Consequently, the condition  $k > 2 -\frac{2}{\alpha}$ in Theorem~\ref{thm:EntropyConvergence} holds for any  $\alpha > 0$.

The entropies of $r_\infty$ are computed explicitly using closed-form expressions for the entropies of the gamma distribution; see \cite[Proposition~3.2]{bodnarchuk2024properties} and also \cite[Proposition~1.15]{entropybook}.
\end{proof}

\subsection{Asymptotic behavior of the entropies of the squared Bessel process as $t \to \infty$}

We now consider the limiting case $b = 0$ of the CIR process, known as the \emph{squared Bessel process}; see, e.g., \cite{GJY03,PitmanYor} or \cite[Chapter XI]{RevuzYor} for details.  
This process $Y = \{Y_t,\, t \ge 0\}$ is defined as the unique strong solution to the stochastic differential equation
\begin{equation}\label{bes}
dY_t = a\, dt + \sigma \sqrt{Y_t}\, dW_t, \quad t \ge 0,
\end{equation}
where $W = \{W_t,\, t \ge 0\}$ is a Wiener process. As in subsection~\ref{sec:cir}, we assume that $Y_0 > 0$, $a > 0$, $\sigma > 0$, and that the Feller condition \eqref{feller} is satisfied.
This condition ensures that the process $Y$ remains almost surely strictly positive, similarly to the CIR process.

The probability density function of the squared Bessel process $Y_t$ is given by
\begin{equation}\label{pdf-bessel}
g_t(x) = \frac{1}{2} \left(\frac{x}{Y_0}\right)^{\nu/2}
\exp\left\{-\frac{2(x+Y_0)}{\sigma^2 t}\right\}
I_{\nu} \left(\frac{4 \sqrt{x Y_0}}{\sigma^2 t}\right) \ind_{x > 0},
\end{equation}
where, as before, $\nu = \frac{2a}{\sigma^2} - 1$; see, e.g., \cite[Chapter XI, Corollary 1.4]{RevuzYor}.

In contrast to the CIR process $X$, the squared Bessel process $Y$ is \emph{non-ergodic}. Since $I_\nu(0) = 0$, it follows that
\[
g_t(x) \to 0, \quad \text{as } t \to \infty.
\]
Hence, the squared Bessel process does not possess a limiting (stationary) distribution.

\begin{theorem}
\label{thm:bessel_entropy_convergence}
Let $Y = \{Y_t, t \ge 0\}$ be the squared Bessel process \eqref{bes} with $Y_0 > 0$ and parameters $a > 0$ and $\sigma > 0$ satisfying the Feller condition \eqref{feller}. Then the following limits hold:
\begin{gather*}
    \lim_{t\to\infty}\entropy_S(Y_t)
  = 
  \lim_{t\to\infty}\entropy_R(Y_t; \alpha)
  = 
  \lim_{t\to\infty}\entropy_{GR}(Y_t; \alpha, \beta )
  =
  \lim_{t\to\infty}\entropy_{GR}(Y_t; \alpha)
  = \infty,
  \\
  \lim_{t\to\infty}\entropy_T(Y_t; \alpha)
  =
  \begin{cases}
  \infty, & \text{if } \alpha < 1, \\
  \frac{1}{\alpha-1}, &  \text{if } \alpha > 1,
  \end{cases}
  \\
  \lim_{t\to\infty}\entropy_{SM}(Y_t; \alpha, \beta )
  =
  \begin{cases}
  \infty, & \text{if } \beta < 1, \\
  \frac{1}{\beta-1}, &  \text{if } \beta > 1.
  \end{cases}
\end{gather*}
\end{theorem}

\begin{proof}
The proof follows the same approach as in Theorem~\ref{thm:CIR_entropy_convergence}. Specifically, we analyze the asymptotic behavior of each entropy using a representation of $Y_t$ as a scaled noncentral chi-squared variable. Comparing the PDF \eqref{pdf-bessel} of $Y_t$ with the PDF \eqref{pdf-noncentral} of a noncentral chi-squared random variable $X_{k,\lambda}$ leads to the following distributional representation:
\[
Y_t \eqd  c_0(t) X_{4a/\sigma^2, \lambda_0(t)},
\]
where
\[
c_0(t) = \frac{\sigma^2 t}{4}, 
\qquad 
\lambda_0(t) = \frac{4 Y_0}{\sigma^2 t}.
\]
Observe that as $t \to \infty$,
\[
c_0(t) \to \infty, 
\qquad 
\lambda_0(t) \to 0.
\]

For the Shannon entropy, Lemma~\ref{l:entropy-scaling} yields
\begin{equation}\label{shannon-bessel}
\entropy_S(Y_t)
= \entropy_S \left(c_0(t) X_{4a/\sigma^2, \lambda_0(t)}\right)
= \entropy_S \left(X_{4a/\sigma^2, \lambda_0(t)}\right)
+ \log c_0(t).
\end{equation}
By Theorem~\ref{thm:EntropyConvergence}, the first term on the right-hand side of \eqref{shannon-bessel} converges to the finite limit $\entropy_S(X_{4a/\sigma^2})$ as $t \to \infty$. Since $\log c_0(t) \to \infty$, we conclude that $\entropy_S(Y_t) \to \infty$.

The same reasoning applies to $\entropy_R(Y_t; \alpha)$, $\entropy_{GR}(Y_t; \alpha, \beta)$, and $\entropy_{GR}(Y_t; \alpha)$, which completes the first part of the theorem.

For the Tsallis entropy, Lemma~\ref{l:entropy-scaling} gives
\begin{equation}\label{tsallis-bessel}
\entropy_T(Y_t; \alpha) 
= \entropy_T \left(c_0(t) X_{4a/\sigma^2, \lambda_0(t)}; \alpha\right)
= c_0^{1 - \alpha}(t) \left(\entropy_T \left(X_{4a/\sigma^2, \lambda_0(t)};\alpha \right) + \frac{1}{1-\alpha}\right)
- \frac{1}{1-\alpha}.
\end{equation}
Since $c_0(t) \to \infty$, we have
\[
c_0^{1 - \alpha}(t) \to 
 \begin{cases}
  \infty, & \text{if } \alpha < 1, \\
  0, &  \text{if } \alpha > 1,
  \end{cases}
  \qquad \text{as } t \to \infty.
\]
Furthermore, by Theorem~\ref{thm:EntropyConvergence}, the entropy 
$\entropy_T(X_{4a/\sigma^2, \lambda_0(t)}; \alpha)$ converges to the finite limit \linebreak
$\entropy_T(X_{4a/\sigma^2}; \alpha)$ as $t \to \infty$.
Using \cite[Proposition 3.6]{bodnarchuk2024properties}, we compute
\[
\entropy_T\left(X_{4a/\sigma^2}; \alpha\right) + \frac{1}{1 - \alpha}
= \frac{1}{1 - \alpha} \cdot
    \frac{2^{1 - \alpha} \alpha^{\alpha\left(1 - \frac{2a}{\sigma^2}\right) - 1} \Gamma\left(\alpha\left(\frac{2a}{\sigma^2} - 1\right) + 1\right)}{\Gamma^\alpha\left(\frac{2a}{\sigma^2}\right)},
\]
which is positive if and only if $\alpha < 1$.
Thus, \eqref{tsallis-bessel} implies that $\entropy_T(Y_t; \alpha) \to \infty$ if $\alpha < 1$, and $\entropy_T(Y_t; \alpha) \to \frac{1}{\alpha - 1}$ if $\alpha > 1$.

The statement for the Sharma--Mittal entropy $\entropy_{SM}(Y_t; \alpha, \beta)$ follows by a similar argument.
\end{proof}
\begin{remark} It is possible to characterize the asymptotic behavior of those entropies that tend to infinity, as normal, but the behavior of Tsallis entropy $\entropy_T(Y_t; \alpha)$ for $\alpha>1$ and Sharma-Mittal entropy  $\entropy_{SM}(Y_t; \alpha, \beta )$ for $\beta>1$ as anomalous, in some sense. This gives some idea of the different behavior of entropies. Another anomalies of entropies (of Poisson distribution) are considered in \cite{finkel}.
    
\end{remark}
\subsection{Convergence of the entropies of the CIR process to those of the squared Bessel process}

We now show that, as the mean-reversion parameter $b$ tends to zero, the entropy functionals of the CIR process converge to those of the squared Bessel process with matching parameters and initial condition.

\begin{theorem}
Let $r = \{r_t, t \ge 0\}$ be the CIR process \eqref{CIRSDE} with parameters $a > 0$, $b > 0$, and $\sigma > 0$ satisfying the Feller condition \eqref{feller}. Let $Y = \{Y_t, t \ge 0\}$ denote the squared Bessel process \eqref{bes} with the same parameters $a$ and $\sigma$, and with the same initial condition $Y_0 = r_0 > 0$. Then, for all $t > 0$, the following convergences hold as $b \downarrow 0$:
\begin{align*}
\entropy_S(r_t) &\to \entropy_S(Y_t), 
  \\
\entropy_R(r_t; \alpha) &\to \entropy_R(Y_t; \alpha),
  \\
\entropy_{GR}(r_t; \alpha, \beta )
  &\to \entropy_{GR}(Y_t; \alpha, \beta ),
  \\
\entropy_{GR}(r_t; \alpha) &\to \entropy_{GR}(Y_t; \alpha),
  \\
\entropy_T(r_t; \alpha) &\to \entropy_T(Y_t; \alpha),
  \\
\entropy_{SM}(r_t; \alpha, \beta ) &\to \entropy_{SM}(Y_t; \alpha, \beta).
\end{align*}
\end{theorem}

\begin{proof}
In the proofs of Theorems \ref{thm:CIR_entropy_convergence} and \ref{thm:bessel_entropy_convergence}, it was shown that the one-dimensional distributions of the CIR and Bessel processes are scaled noncentral chi-squared distributions. Specifically,
\[
r_t \eqd c(t) X_{4a/\sigma^2, \lambda(t)},
\quad
Y_t \eqd  c_0(t) X_{4a/\sigma^2, \lambda_0(t)},
\]
where
\[
 c(t) = \frac{\sigma^2}{4b}\left( 1 - e^{-bt} \right) \to \frac{\sigma^2 t}{4} = c_0(t), 
\quad 
\lambda(t) = \frac{r_0 e^{-bt}}{c(t)} \to \frac{4 r_0}{\sigma^2 t} = \lambda_0(t),
\quad \text{as } b \downarrow 0.
\]
Therefore, by arguments analogous to those used in the proofs of Theorems \ref{thm:CIR_entropy_convergence} and \ref{thm:bessel_entropy_convergence}, and employing Remark~\ref{rem:conv-chi-sq} in place of Theorem \ref{thm:EntropyConvergence}, we obtain the stated convergence results.

For instance, in the case of Shannon entropy, we have:
\begin{align*}
\entropy_S(r_t)
&= \entropy_S \left(c(t) X_{4a/\sigma^2, \lambda(t)}\right)
= \entropy_S \left(X_{4a/\sigma^2, \lambda(t)}\right)
+ \log c(t)
\\
&\to \entropy_S \left(X_{4a/\sigma^2, \lambda_0(t)}\right)
+ \log c_0(t)
= \entropy_S \left(c_0(t) X_{4a/\sigma^2, \lambda_0(t)}\right)
= \entropy_S(Y_t),
\quad \text{as } b \downarrow 0.
\end{align*}

The convergence of the other entropy functionals follows analogously.
\end{proof}

\appendix
\section{Auxiliary results}

\subsection{Inequalities for the modified Bessel function of the first kind and the PDF of the noncentral chi-squared distribution}
\label{app:inequality}

In this appendix, we present several inequalities for the modified Bessel function of the first kind $I_\nu(x)$, which appears in the PDF of the noncentral chi-squared distribution. We also derive the corresponding inequalities for the PDF itself, which are instrumental in our proofs.

Let $\nu > -1$ and $x \in \mathbb{R}$. The \emph{modified Bessel function of the first kind} of order $\nu$ is defined by
\begin{equation}\label{bessel-func}
    I_{\nu}(x) = \sum_{m = 0}^{\infty} \frac{1}{m! \, \Gamma(m + \nu + 1)} \left(\frac{x}{2}\right)^{2m + \nu}.
\end{equation}
Throughout this paper, we only consider the case $x > 0$, where the function $I_{\nu}(x)$ is real-valued for any $\nu > -1$. For additional properties and details, we refer the reader to the classical monograph~\cite{Watson}.

According to \cite[formula (6.25)]{Luke1972}, for $x > 0$ and $\nu > -\frac{1}{2}$, the function $I_\nu(x)$ satisfies the following two-sided inequality:
\begin{equation}\label{ineq-luke}
    e^{-x} < \Gamma(\nu + 1)\left( \frac{2}{x} \right)^\nu e^{-x}I_\nu (x) < \frac{1}{2}\left(1 + e^{-2x}\right).
\end{equation}

The following two lemmas are essential in establishing the convergence theorems for entropies in Section~\ref{sec:conv-shi-sq}. First, we bound $I_\nu(x)$ from both sides using functions that are more analytically tractable. Then we apply these bounds to the PDF of the noncentral chi-squared distribution to eliminate the Bessel function from the expression. These inequalities are used repeatedly in Appendix~\ref{app:integrals} to construct integrable majorants for the entropy integrands, thereby enabling the application of the dominated convergence theorem.

\begin{lemma}[Bounds for the modified Bessel function of the first kind]\label{lemma:bessel_bounds}
For $x > 0$ and $\nu > -\frac{1}{2}$, the modified Bessel function of the first kind satisfies the inequalities
\begin{align}
    \left( \frac{x}{2} \right)^\nu \frac{1}{\Gamma(\nu + 1)} <
    I_\nu(x) < \left( \frac{x}{2} \right)^\nu \frac{1}{\Gamma(\nu + 1)} e^x. \label{lemma:bessel_bounds_inequalities}
\end{align}
\end{lemma}

\begin{proof}
Rewriting inequality~\eqref{ineq-luke}, we obtain
\begin{align*}
    \left( \frac{x}{2} \right)^\nu \frac{1}{\Gamma(\nu + 1)} <
    I_\nu(x) < \left( \frac{x}{2} \right)^\nu \frac{1}{\Gamma(\nu + 1)} \cosh(x).
\end{align*}
Since $\cosh(x) < e^x$ for all $x > 0$, the desired result follows.
\end{proof}

\begin{lemma}[Bounds for the PDF of the noncentral chi-squared distribution]\label{lemma:noncentralpdfbounds}
Let $f_{X_{k,\lambda}}(x)$ denote the PDF of the noncentral chi-squared distribution with $k > 1$ degrees of freedom and noncentrality parameter $\lambda > 0$. Then, for all $x > 0$, the following bounds hold:
\begin{align*}
    \frac{1}{2^{k/2}} e^{-(x + \lambda)/2} x^{k/2 - 1} \frac{1}{\Gamma(k/2)}
    \leq f_{X_{k,\lambda}}(x)
    \le \frac{1}{2^{k/2}} e^{-x/4+\lambda/2} x^{k/2 - 1} \frac{1}{\Gamma(k/2)}.
\end{align*}
\end{lemma}

\begin{proof}
Recall that the PDF of the noncentral chi-squared distribution is given by
\begin{align*}
    f_{X_{k,\lambda}}(x) = \frac{1}{2} e^{-(x + \lambda)/2} \left(\frac{x}{\lambda}\right)^{k/4 - 1/2} I_{k/2 - 1}(\sqrt{\lambda x})
\end{align*}
for $k, \lambda > 0$.

Applying the bounds from~\eqref{lemma:bessel_bounds_inequalities} to $I_{k/2 - 1}(\sqrt{\lambda x})$, we obtain:
\begin{multline}\label{lemma:noncentralpdfbounds:proof:1}
    \frac{1}{2} e^{-(x + \lambda)/2} \left(\frac{x}{\lambda}\right)^{k/4 - 1/2} \left( \frac{\sqrt{\lambda x}}{2} \right)^{k/2 - 1} \frac{1}{\Gamma(k/2)} \\
    \leq f_{X_{k,\lambda}}(x) \leq
    \frac{1}{2} e^{-(x + \lambda)/2} \left(\frac{x}{\lambda}\right)^{k/4 - 1/2} \left( \frac{\sqrt{\lambda x}}{2} \right)^{k/2 - 1} \frac{1}{\Gamma(k/2)} e^{\sqrt{\lambda x}}.
\end{multline}
Simplifying~\eqref{lemma:noncentralpdfbounds:proof:1}, we find:
\[
    \frac{1}{2^{k/2}} e^{-(x + \lambda)/2} x^{k/2 - 1} \frac{1}{\Gamma(k/2)}
    \leq f_{X_{k,\lambda}}(x)
    \leq \frac{1}{2^{k/2}} e^{-(x + \lambda)/2 + \sqrt{\lambda x}} x^{k/2 - 1} \frac{1}{\Gamma(k/2)}.
\]
Finally, applying the inequality $-x/4 + \sqrt{\lambda x} \leq \lambda/2$, which holds for all $x, \lambda > 0$, yields
\[
    \frac{1}{2^{k/2}} e^{-(x + \lambda)/2} x^{k/2 - 1} \frac{1}{\Gamma(k/2)}
    \leq f_{X_{k,\lambda}}(x)
    \leq \frac{1}{2^{k/2}} e^{-x/4 + \lambda/2} x^{k/2 - 1} \frac{1}{\Gamma(k/2)},
\]
as claimed.
\end{proof}

\begin{remark}
We emphasize that Lemma~\ref{lemma:noncentralpdfbounds} requires the condition $k > 1$ on the degrees of freedom. This condition arises from the requirement $\nu > -\frac{1}{2}$ imposed by inequality~\eqref{ineq-luke} and, consequently, by Lemma~\ref{lemma:bessel_bounds}, which concern the modified Bessel function of the first kind. In the PDF~\eqref{pdf-noncentral} of the noncentral chi-squared distribution, this function appears with parameter $\nu = \frac{k}{2} - 1$. Therefore, the condition $k > 1$ ensures that $\nu > -\frac{1}{2}$ holds.
\end{remark}

\subsection{Convergence of improper integrals involving PDF of noncentral chi-squared distribution}
\label{app:integrals}

In this appendix, we present, for reference, several formulas and convergence results for improper integrals that arise in the proofs in Section~\ref{sec:conv-shi-sq}.

It is well known (see, e.g., \cite[formula 4.358.5]{Gradshteyn}) that for all $\nu > 0$ and $\mu > 0$,
\begin{equation}\label{gammaintderiv}
\int_0^\infty x^{\nu - 1} e^{-\mu x} \log x \, dx
= \mu^{-\nu} \Gamma(\nu) \left( \psi(\nu) - \log \mu \right),
\end{equation}
where $\psi(\nu)$ denotes the digamma function.

\begin{lemma}\label{l:intlog}
The integral
\[
\int_0^\infty x^{\nu - 1} e^{-\mu x} |\log x|\, dx
\]
converges if and only if $\nu > 0$ and $\mu > 0$.
\end{lemma}

\begin{proof}
We split the integral as follows:
\[
\int_0^\infty x^{\nu - 1} e^{-\mu x} \abs{\log x} \, dx 
= -\int_0^1 x^{\nu - 1} e^{-\mu x} \log x \, dx
+ \int_1^\infty x^{\nu - 1} e^{-\mu x} \log x \, dx.
\]
Each term is finite when $\nu > 0$ and $\mu > 0$, as follows from~\eqref{gammaintderiv}.

As $x \downarrow 0$, the integrand behaves like $x^{\nu - 1} \log x$, which is integrable near zero if and only if $\nu > 0$. For $x > 1$, the exponential decay $e^{-\mu x}$ ensures integrability for any $\nu \in \mathbb{R}$ provided that $\mu > 0$.
If $\mu \le 0$, then as $x \to \infty$, the integrand behaves like $x^{\nu - 1} \log x$, and the integral diverges for all $\nu > 0$. This shows that both conditions are necessary.
\end{proof}

\begin{lemma}\label{l:integrals}
Let $\alpha > 0$, $k > 1$, and $\lambda > 0$, and let $f_{X_k}$ and $f_{X_{k,\lambda}}$ denote the PDFs of the central and noncentral chi-squared distributions, as defined in \eqref{pdf-central} and \eqref{pdf-noncentral}, respectively. Then:
\begin{enumerate}[$(i)$]
\item
The integral
$\int_{\mathbb{R}} f_{X_{k,\lambda}}^\alpha(x) \, dx$
converges if and only if
\begin{equation}\label{condition}
k > 2 - \frac{2}{\alpha}.
\end{equation}
Moreover, under condition \eqref{condition}, the following convergence holds:
\begin{equation}
\lim_{\lambda \to 0} \int_0^{\infty} f_{X_{k,\lambda}}^\alpha(x) \, dx = \int_0^{\infty} f_{X_k}^\alpha(x) \, dx.
\label{convergenceOfIntegralOfPower}
\end{equation}

\item
Under the same condition \eqref{condition}, the integral
$\int_{\mathbb{R}} f_{X_{k,\lambda}}^\alpha(x) \log f_{X_{k,\lambda}}(x) \, dx$
also converges, and
\begin{equation}
\lim_{\lambda \to 0} \int_0^{\infty} f_{X_{k,\lambda}}^\alpha(x)\log f_{X_{k,\lambda}}(x) \, dx = 
\int_0^{\infty} f_{X_k}^\alpha(x)\log f_{X_k}(x) \, dx.
\label{convergenceOfIntegralOfPowerTimesLogOfPDF}
\end{equation}
\end{enumerate}
\end{lemma}

\begin{proof}
$(i)$
Consider the integral
$\int_{\mathbb{R}} f_{X_{k,\lambda}}^\alpha(x) \, dx$.
Using Lemma~\ref{lemma:noncentralpdfbounds}, raising the bounding expressions to the power $\alpha$, and integrating, we obtain:
\begin{equation} \label{integrableupperboundforpower1}
\left(\frac{1}{2^{k/2}} e^{-(x + \lambda)/2} x^{k/2 - 1} \frac{1}{\Gamma(k/2)}\right)^\alpha
\leq f_{X_{k,\lambda}}^\alpha(x)
\leq \left(\frac{1}{2^{k/2}} e^{-x/4 + \lambda/2} x^{k/2 - 1} \frac{1}{\Gamma(k/2)}\right)^\alpha.
\end{equation}
Hence,
\begin{multline} \label{integrableupperboundforpower2}
\int_0^{\infty} \left(\frac{1}{2^{k/2}} e^{-(x + \lambda)/2} x^{k/2 - 1} \frac{1}{\Gamma(k/2)}\right)^\alpha dx
\leq \int_0^{\infty} f_{X_{k,\lambda}}^\alpha(x) \, dx \\
\leq \int_0^{\infty} \left(\frac{1}{2^{k/2}} e^{-x/4 + \lambda/2} x^{k/2 - 1} \frac{1}{\Gamma(k/2)}\right)^\alpha dx.
\end{multline}

Next, apply the substitutions $u = \frac{\alpha x}{2}$ and $v = \frac{\alpha x}{4}$ to evaluate the bounding integrals. Define
\begin{gather*}
C_1 = \frac{2}{\alpha} \left(\frac{e^{-\lambda /2} (2/\alpha)^{k/2 - 1}}{2^{k/2} \Gamma(k/2)}\right)^\alpha
= \frac{2}{\alpha} \left(\frac{e^{-\lambda /2}}{2 \alpha^{k/2 - 1} \Gamma(k/2)}\right)^\alpha,
\\
C_2 = \frac{4}{\alpha} \left(\frac{e^{\lambda /2} (4/\alpha)^{k/2 - 1}}{2^{k/2} \Gamma(k/2)}\right)^\alpha
= \frac{4}{\alpha} \left(\frac{e^{\lambda /2}\, 2^{k/2 - 2}}{\alpha^{k/2 - 1}\Gamma(k/2)}\right)^\alpha,
\end{gather*}
which are independent of $x$. Then the following two-sided bounds hold:
\[
C_1 \int_0^{\infty} u^{\alpha(k/2 - 1)} e^{-u} \, du
\leq \int_0^{\infty} f_{X_{k,\lambda}}^\alpha(x) \, dx
\leq C_2 \int_0^{\infty} v^{\alpha(k/2 - 1)} e^{-v} \, dv.
\]
These are gamma integrals and converge if and only if $\alpha(k/2 - 1) > -1$, i.e., $k > 2 - \frac{2}{\alpha}$.

To prove \eqref{convergenceOfIntegralOfPower}, note that pointwise convergence of the integrand follows from \eqref{chi_squared_convergence}. Furthermore, the bounds in \eqref{integrableupperboundforpower1}–\eqref{integrableupperboundforpower2} provide an integrable upper bound independent of $\lambda$ for $\lambda < 1$, since the upper bound is increasing in $\lambda$. Evaluating at $\lambda = 1$ yields a uniform integrable majorant. Therefore, the dominated convergence theorem implies \eqref{convergenceOfIntegralOfPower}.

\medskip

$(ii)$
Consider the integral
\[
\int_0^{\infty} f_{X_{k, \lambda}}^\alpha(x)\log f_{X_{k, \lambda}}(x)\, dx.
\]
We again use the upper bound \eqref{integrableupperboundforpower1} for $f_{X_{k,\lambda}}^\alpha(x)$.

To estimate the logarithmic term, we apply the inequality: for $a \le b \le c$, it holds that $|\log b| \le |\log a| + |\log c|$. Applying this to the upper and lower bounds of $f_{X_{k,\lambda}}$ gives
\begin{align}
\MoveEqLeft
\left|\log \left( \frac{1}{2} e^{-(x + \lambda)/2} \left(\frac{x}{\lambda}\right)^{k/4 - 1/2} I_{k/2 - 1}(\sqrt{\lambda x}) \right)\right| 
\notag\\
&\leq \left| -\frac{x}{2} - \frac{\lambda}{2} - \frac{k}{2} \log 2 + \left( \frac{k}{2} - 1 \right) \log x - \log \Gamma(k/2) \right| \notag \\
&\quad + \left| -\frac{x}{4} + \frac{\lambda}{2} - \frac{k}{2} \log 2 + \left( \frac{k}{2} - 1 \right) \log x - \log \Gamma(k/2) \right| \notag \\
&\leq \frac{3}{4} x + \lambda +  k\log 2 + |(k - 2)\log x| + |\log \Gamma(k/2)|. \label{logupperbound}
\end{align}

Combining \eqref{integrableupperboundforpower1} and \eqref{logupperbound}, we obtain:
\begin{align}
\MoveEqLeft[1]
\int_0^{\infty} \abs{f_{X_{k, \lambda}}^\alpha(x)\log f_{X_{k, \lambda}}(x)} \, dx  
\notag\\
&\leq \int_0^{\infty} \left(\frac{1}{2^{k/2}} e^{-\frac{x}{4} + \lambda/2} x^{k/2-1} \frac{1}{\Gamma(k/2)}\right)^\alpha \left(\frac{3}{4} x + \lambda +  k\log 2 + |(k - 2)\log x| + |\log \Gamma(k/2)|\right) dx
\notag\\
&= C_3 \int_0^{\infty} e^{-\frac{\alpha x}{4}} x^{\alpha(k/2-1)} \left(\frac{3}{4} x + \lambda +  k\log 2 + |(k - 2)\log x| + |\log \Gamma(k/2)|\right) dx
\notag\\
&= \frac{3}{4} C_3 \int_0^{\infty} e^{-\frac{\alpha x}{4}} x^{\alpha(k/2-1)+1} dx
+ C_3 \left(\lambda +  k\log 2 + |\log \Gamma(k/2)|\right)\int_0^{\infty} e^{-\frac{\alpha x}{4}} x^{\alpha(k/2-1)} dx
\notag\\
&\quad + C_3 |k - 2|\int_0^{\infty} e^{-\frac{\alpha x}{4}} x^{\alpha(k/2-1)} |\log x| dx,
\label{logintegrand}
\end{align}
where
$C_3 = \left(\frac{e^{\lambda/2}}{2^{k/2}\Gamma(k/2)}\right)^\alpha$.
The first two integrals on the right-hand side of \eqref{logintegrand} are of gamma type and finite under condition \eqref{condition}. According to Lemma~\ref{l:intlog}, the third integral is also finite under this condition.

To establish \eqref{convergenceOfIntegralOfPowerTimesLogOfPDF}, note that the integrand converges pointwise:
\[
\lim_{\lambda \to 0} f_{X_{k,\lambda}}^\alpha(x)\log f_{X_{k,\lambda}}(x) = f_{X_k}^\alpha(x)\log f_{X_k}(x),
\]
as follows from \eqref{chi_squared_convergence}. Furthermore, the integrable upper bound in \eqref{logintegrand} is valid for all $\lambda < 1$, and grows monotonically with $\lambda$. Setting $\lambda = 1$ yields a uniform majorant. Applying the dominated convergence theorem completes the proof.
\end{proof}

\begin{remark}
Let $\alpha > 0$, and let $f_{X_{k,\lambda}}$ and $f_{X_{k,\lambda_0}}$ denote the probability density functions of noncentral chi-squared distributions with $k > 1$ degrees of freedom and noncentrality parameters $\lambda > 0$ and $\lambda_0 > 0$, respectively. Then, under condition~\eqref{condition}, the dominated convergence theorem, applied similarly to the proof of Lemma~\ref{l:integrals}, yields the following limits:
\begin{gather*}
\lim_{\lambda \to \lambda_0} \int_0^{\infty} f_{X_{k,\lambda}}^\alpha(x) \, dx = \int_0^{\infty} f_{X_{k,\lambda_0}}^\alpha(x) \, dx, \\
\lim_{\lambda \to \lambda_0} \int_0^{\infty} f_{X_{k,\lambda}}^\alpha(x)\log f_{X_{k,\lambda}}(x) \, dx = 
\int_0^{\infty} f_{X_{k,\lambda_0}}^\alpha(x)\log f_{X_{k,\lambda_0}}(x) \, dx.
\end{gather*}
\end{remark}

\providecommand{\bysame}{\leavevmode\hbox to3em{\hrulefill}\thinspace}
\providecommand{\MR}{\relax\ifhmode\unskip\space\fi MR }
\providecommand{\MRhref}[2]{%
  \href{http://www.ams.org/mathscinet-getitem?mr=#1}{#2}
}
\providecommand{\href}[2]{#2}

\end{document}